\documentclass[12pt,reqno]{amsart}

\usepackage{multirow}
\usepackage{hhline}

\usepackage{mathtools}
\usepackage{times}
\usepackage[T1]{fontenc}
\usepackage{mathrsfs}
\usepackage{latexsym}
\usepackage[dvips]{graphics}
\usepackage[titletoc, title]{appendix}
\usepackage{epsfig}
\usepackage{amssymb}
\usepackage{amsmath,amsfonts,amsthm,amssymb,amscd}
\usepackage{color}
\usepackage{hyperref}
\usepackage{amsmath}

\usepackage{color}
\usepackage{breakurl}

\usepackage{comment}
\newcommand{\bburl}[1]{\textcolor{blue}{\url{#1}}}

\newcommand{\be}{\begin{equation}}
\newcommand{\ee}{\end{equation}}
\newcommand{\seqnum}[1]{\href{https://urldefense.com/v3/__https://oeis.org/*1*7D*7B*5Crm__;IyUlJQ!!DZ3fjg!5yDnHeGSfGljEqbpuRw46hLMiV_jaL8rn2Yk7DqDl12SpLF1Wq4c5kckywJ55Db23mDqXkwmG3-qhLCpIIaQw7Y$   \underline{#1}}}

\newtheorem{thm}{Theorem}[section]

\newtheorem{claim}[thm]{Claim}
\newtheorem{lem}[thm]{Lemma}
\newtheorem{prop}[thm]{Proposition}

\newtheorem{defi}[thm]{Definition}

\numberwithin{equation}{section}

\newcommand{\Mod}[1]{\ \mathrm{mod}\ #1}

\begin{document}

\title[Linear Recurrences of Order at Most Two in Nontrivial Divisors]{Linear Recurrences of Order at Most Two in Nontrivial Small Divisors and Large Divisors}

\author[Chu]{H\`ung Vi\d{\^e}t Chu}
\email{\textcolor{blue}{\href{mailto:hungchu2@illinois.edu}{hungchu2@illinois.edu}}}
\address{Department of Mathematics, University of Illinois at Urbana-Champaign, Urbana, IL 61820, USA}

\author[Le]{Kevin Huu Le}
\email{\textcolor{blue}{\href{mailto:kevinhle@tamu.edu}{kevinhle@tamu.edu}}}
\address{Mathematics, Texas A\&M University College Station, TX 77843, USA}

\author[Miller]{Steven J. Miller}
\email{\textcolor{blue}{\href{mailto:sjm1@williams.edu}{sjm1@williams.edu}},
\textcolor{blue}{\href{Steven.Miller.MC.96@aya.yale.edu}{Steven.Miller.MC.96@aya.yale.edu}}}
\address{Department of Mathematics and Statistics, Williams College, Williamstown, MA 01267, USA}

\author[Qiu]{Yuan Qiu}
\email{\textcolor{blue}{\href{mailto:yq1@williams.edu}{ yq1@williams.edu}}}
\address{Williams College, Williamstown, MA 01267, USA}

\author[Shen]{Liyang Shen}
\email{\textcolor{blue}{\href{mailto:Liyang.Shen@nyu.edu}{Liyang.Shen@nyu.edu}}}
\address{Courant Institute of Mathematical Sciences, New York University, New York, NY 10012, USA}

\begin{abstract} 
For each positive integer $N$, define 
$$S'_N \ =\ \{1 < d < \sqrt{N}: d|N\}\mbox{ and }L'_N \ =\ \{\sqrt{N} < d < N : d|N\}.$$
Recently, Chentouf characterized all positive integers $N$ such that the set of small divisors $\{d\le \sqrt{N}: d|N\}$  satisfies a linear recurrence of order at most two. We nontrivially extend the result by excluding the trivial divisor $1$ from consideration, which dramatically increases the analysis complexity. Our first result characterizes all positive integers $N$ such that $S'_N$ satisfies a linear recurrence of order at most two. Moreover, our second result characterizes all positive $N$ such that $L'_N$ satisfies a linear recurrence of order at most two, thus extending considerably a recent result that characterizes $N$ with $L'_N$ being in an arithmetic progression. 
\end{abstract}



\thanks{}

\maketitle

\tableofcontents

\section{Introduction}
For a positive integer $N$, the set of small divisors of $N$ is  
$$S_N\ := \ \{d\,:\, 1\le d\le \sqrt{N}, d\mbox{ divides }N\}.$$
Since the case $N = 1$ is trivial, we assume throughout the paper that $N > 1$. 
In 2018, Iannucci characterized all positive integers $N$ whose $S_N$ forms an arithmetic progression (or AP, for short). Iannucci's key idea was to show that if $S_N$ forms an AP, then the size $|S_N|$ cannot exceed $6$. Observing that the trivial divisor $1$ plays an important role in Iannucci's proofs (see \cite[Lemma 3 and Theorem 4]{Ian}), Chu \cite{Chu2} excluded both $1$ and $\sqrt{N}$ from the definition of $S_N$ to obtain a more general theorem that characterizes all $N$ whose
$$S'_N\ :=\ \{d\,:\, 1 < d< \sqrt{N}, d\mbox{ divides }N\}$$
is in an AP. Interestingly, with the trivial divisor $1$ excluded, \cite[Theorem 1.1]{Chu1} still gives that $|S'_N|\le 5$.  
Recently, Chentouf generalized Iannucci's result from a different perspective by characterizing all $N$ whose $S_N$ satisfies a linear recurrence of order at most two. In particular, for each tuple $(u, v, a, b) \in \mathbb{Z}^4$, there is an integral linear recurrence, denoted by $U(u, v, a, b)$, of order at most two, given by
$$n_i\ =\ \begin{cases}u&\mbox{ if }i = 1,\\ v &\mbox{ if }i = 2,\\ an_{i-1} + bn_{i-2}&\mbox{ if }i\ge 3.\end{cases}$$
Noting that the appearance of the trivial divisor $1$ contributes nontrivially to the proof of \cite[Theorem 3, Lemma 8, Theorem 10]{C}, we generalize Chentouf's result in the same manner as \cite[Theorem 1.1]{Chu1} generalizes \cite[Theorem 4]{Ian}: we characterize all positive integers $N$ whose $S'_N$ satisfies a linear recurrence of order at most two. 

\begin{defi}\normalfont
A positive integer $N$ is said to be small recurrent if $S'_N$ satisfies a linear recurrence of order at most two. When $|S'_N|\le 2$, $N$ is vacuously small recurrent.
\end{defi}

\begin{thm}\label{m1}
Let $p, q, r$ denote prime numbers such that $p < q< r$ and $k$ be some positive integer. 
A positive integer $N > 1$ is small recurrent if and only if $N$ belongs to one of the following forms.
\begin{enumerate}
    \item $N = p^k$ for some $k\ge 1$. In this case, $S'_N = \{p, p^2, \ldots, p^{\lfloor (k-1)/2\rfloor}\}$ satisfies $U(p, p^2, p, 0)$.
    \item $N = p^kq$ or $N = pq^k$ for some $1\le k\le 3$. A restriction for $N = p^3q$ is that either $p < q < p^2$ or $p^3 < q$.
    \item $N = p^kq$ for some $k\ge 4$ and $q > p^k$. In this case, $S'_N = \{p, p^2, p^3, \ldots, p^k\}$ satisfies $U(p, p^2, p, 0)$.
    \item $N = p^k q$ for some $k\ge 4$ and $\sqrt{q} < p < q$. In this case, $S'_N = \{p, q, p^2, pq, \ldots\}$ satisfies $U(p, q, 0, p)$. 
    \item $N = pq^k$ for some $k\ge 4$ and $p < q$. In this case, $S'_N = \{p, q, pq, q^2, \ldots\}$ satisfies $U(p, q, 0, q)$.
    \item   $N = pq^kr$ for some $k\ge 2$, $p < q$, and $r > pq^k$.
        In this case, $S'_N = \{p, q, pq, q^2, \ldots, pq^{k-1}, q^k, pq^k\}$ satisfies $U(p, q, 0, q)$.
     \item $N = p^2 q^2$ for some $p < q < p^2$. In this case, $S'_N = \{p, q, p^2\}$. 
     \item $N = pqr$ for some $p < q < r$. If $r < pq$, then $S'_N = \{p, q, r\}$. If $r > pq$, then $S'_N = \{p, q, pq\}$.
         \item $N = p^3 q^2$ for some $p^{3/2} < q < p^2$. In this case, $S'_N = \{p, q, p^2, pq, p^3\}$ satisfies $U(p, q, 0, p)$.
    \item $N = p^2qr$, where $p < q < p^2 < r < pq$, $(q^2 - p^3)|(pq - r)$, $(q^2-p^3)|(rq-p^4)$, and $r =  pq-\sqrt{(q^2-p^3)(p^2-q)}$. In this case, $S'_N = \{p, q, p^2, r, pq\}$ satisfies $U\left(p, q, \frac{p(pq - r)}{q^2 - p^3}, \frac{rq - p^4}{q^2 - p^3}\right)$.
\end{enumerate}
\end{thm}

Next, consider the set of large divisors of $N$
$$L_N \ := \ \{d\,:\, d \ge \sqrt{N}, d\mbox{ divides }N\}\mbox{ and }L'_N \ :=\ \{d\,:\, \sqrt{N}< d< N, d\mbox{ divides } N\}.$$
The second result of this paper is the characterization of all positive integers $N$ whose $L'_N$ satisfies a linear recurrence of order at most two. This considerably extends \cite[Theorem 1]{Chu1}.

\begin{defi}\normalfont
A positive integer $N$ is said to be large recurrent if $L'_N$ satisfies a linear recurrence of order at most two. When $|L'_N|\le 2$, $N$ is vacuously large recurrent.
\end{defi}
\begin{thm}\label{m2}
Let $p, q, r$ denote prime numbers such that $p < q< r$ and $k$ be some positive integer. 
A positive integer $N > 1$ is large recurrent if and only if $N$ belongs to one of the following forms.
\begin{enumerate}
     \item $N=p^k$ for some $k\ge 1$. In this case, $L'_N=\{p^{\lceil(k-1)/2\rceil+1}, p^{\lceil (k-1)/2\rceil+2},\ldots,p^{k-1}\}$ satisfies $U(p^{\lceil (k-1)/2\rceil+1},p^{\lceil (k-1)/2\rceil +2},p,0)$.
    \item $N=p^kq$ for some $k\ge 1$ and $q>p^k$. In this case, $L'_N=\{q,pq,p^2q,\ldots,p^{k-1}q\}$ satisfies $U(q,pq,p,0)$.
    \item $N=p^kq$ for some $k\ge 2$ and $p^{k-1} < q < p^k$. Then $$L'_N \ =\ \{p^k, pq, p^2q, \ldots, p^{k-1}q\}$$
    satisfies $U(p^k, pq, p, 0)$.
     \item $N = p^kq$ some for $k\ge 3$ and $p<q<p^2$.  In this case, 
    $$L'_N \ =\ \begin{cases}
    \{p^{{k/2}+1},p^{{k/2}}q,p^{{k/2}+2},\ldots,p^{k-1}q\} &\mbox{ if }2|k,\\
    \{p^{(k-1)/2}q,p^{{(k+3)/2}},p^{(k+1)/2}q,\ldots,p^{k-1}q\} &\mbox{ if }2\nmid k.
    \end{cases}$$
    Observe that $L'_N$ satisfies $U(p^{{k/2}+1},p^{{k/2}}q,0,p)$ and $U(p^{(k-1)/2}q,p^{{(k+3)/2}},0,p)$ for even and odd $k$, respectively.
    \item $N=p^4q$ with $p^2 < q < p^3$, $(p^5-q^2)|(p^2-q)$, and $(p^5-q^2)|(p^3-q)$. In this case, 
    $L'_N \ =\ \{pq, p^4, p^2q, p^3q\}$.
     \item $N = p^3q^2$ for $p < q < p^2$. In this case, $L'_N = \{q^2, p^2q, pq^2, p^3q, p^2q^2\}$ satisfies $U(q^2, p^2q, 0, p)$.
     \item $N = p^2 q^2$ for some $p < q$. If $p < q < p^2$, then $L'_N = \{q^2, p^2q, pq^2\}$. 
    \item $N = pq^k$ for some $k\ge 2$ and $p < q$.  In this case, 
    $$L'_N \ =\ \begin{cases}
   \{pq^{k\over2},q^{{k\over2}+1},\ldots,q^k\} &\mbox{ if }2|k,\\
    \{q^{k+1\over2},pq^{k+1\over2},\ldots,q^k\} &\mbox{ if }2\nmid k.
    \end{cases}$$
    Observe that $L'_N$ satisfies $U(pq^{k/2},q^{{k/2}+1},0,q)$ and $U(q^{(k+1)/2},pq^{(k+1)/2},0,q)$ for even and odd $k$, respectively.
     \item  $N=pq^kr$ for some $k\ge 1$ and $p<q<pq^k<r$. In this case,  $L'_N=\{r,pr,qr,pqr,q^2r,\ldots,q^kr\}$ satisfies $U(r,pr,0,q)$.
\end{enumerate}
\end{thm}

The paper is structured as follows: Section \ref{pre} studies the case when $N$ has a small number of divisors and establish some preliminary results; Section \ref{small} characterizes small recurrent numbers, while Section \ref{large} characterizes large recurrent numbers.

\section{Preliminaries}\label{pre}
For each $N\in\mathbb{N}$ with the prime factorization $\prod_{i=1}^\ell p_i^{a_i}$, the divisor-counting function is 
\begin{equation}\label{e1}\tau(N)\ :=\ \sum_{d|N} 1\ =\ \prod_{i=1}^\ell (a_i+1).\end{equation}
It is easy to verify that for $N > 1$,
\begin{equation}\label{e5}\tau(N) \ :=\ \begin{cases}2|S'_N| + 3\ =\ 2|L'_N| + 3 &\mbox{ if }N\mbox{ is a square},\\ 2|S'_N| + 2\ =\ 2|L'_N| + 2 &\mbox{ otherwise}.\end{cases}\end{equation}
Using \eqref{e1}, we can characterize all $N$ with $\tau(N)\le 9$; equivalently, $|S'_N|, |L'_N|\le 3$.
\begin{enumerate}
    \item[(i)] If $\tau(N) = 2$ or $3$,  \eqref{e1} gives that $N = p$ or $p^2$ for some prime $p$,
    \item[(ii)] If $\tau(N) = 4$ or $5$,  $N = pq, p^3, p^4$ for some primes $p < q$,
    \item[(iii)] If $\tau(N) = 6$ or $7$, $N = p^5, pq^2, p^2q, p^6$ for some primes $p < q$,
    \item[(iv)] If $\tau(N) = 8$ or $9$, $N = pqr, pq^3, p^3q, p^7, p^2q^2, p^8$ for some primes $p < q < r$.
\end{enumerate}

\subsection{Regarding $S'_N$}

\begin{prop}\label{p2}
If $|S'_N|\ge 2$, then $N$ cannot have two (not necessarily distinct) prime factors $p_1$ and $p_2$ at least $\sqrt{N}$
\end{prop}

\begin{proof}
Suppose that $N$ has two prime factors $p_1$ and $p_2$ at least $\sqrt{N}$. Then $p_1p_2\ge N$ and $p_1p_2$ divides $N$; hence, $N = p_1p_2$, which contradicts that $|S'_N|\ge 2$.
\end{proof}

\begin{prop}\label{p1}
If all elements of $S'_N$ are divisible by some prime $p$ and $|S'_N|\ge 4$, then either $N = p^k$ or $N = p^kq$ for some $k\ge 1$ and some prime $q > p^k$. 
\end{prop}

\begin{proof}
If all divisors (except $1$) of $N$ are divisible by $p$, then $n = p^k$ for some $k\ge 1$. Assume that $N$ has a prime factor $q\neq p$. Then $q\ge \sqrt{N}$. Proposition \ref{p2} implies that $N$ cannot have another prime factor at least $\sqrt{N}$. Hence, $N = p^kq$ for some $k\ge 1$ and $q > p^k$. 
\end{proof}

Let $p < q < r < s$ be distinct prime numbers. Write $S'_N = \{d_2, d_3, d_4, d_5, \ldots\}$. (We start with $d_2$ since the smallest divisor of $N$ is usually denoted by $d_1 = 1$, which is excluded from $S'_N$.)

\begin{lem}\label{l1} Suppose that the first $4$ numbers in $S'_N$ are $p < q < p^2 < r$. If $N$ is small recurrent with $U(p, q, a, b)$, the following hold:
\begin{enumerate}
    \item[i)] $\gcd(a, b) = 1$.
    \item[ii)] if $d_{2i}\in S'_N$, then $p|d_{2i}$; however, if $d_{2i-1}\in S'_N$, then $p\nmid d_{2i-1}$.
    \item[iii)] for $d_i, d_{2i-1}\in S'_N$, we have $\gcd(b, d_i) =  \gcd(a, d_{2i-1}) = 1$.
    \item[iv)] for $d_i, d_{i+1}\in S'_N$, we have $\gcd(d_i, d_{i+1}) = 1$.
    \item[v)] for $d_{2i-1}, d_{2i+1}\in S'_N$, we have $\gcd(d_{2i-1}, d_{2i+1})= 1$.
\end{enumerate}
\end{lem}

\begin{proof}
i) Since $r = ap^2 + bq$, we have $\gcd(a, b)|r$. Observe that $\gcd(a, b) = r$ contradicts $p^2 = aq + bp$. Hence, $\gcd(a, b) = 1$.

ii) We prove by induction. The claim is true for $i\le 2$. Assume that the claim holds for $i = j\ge 2$. We show that it holds for $i = j+1$. Since $p^2 = aq + bp$, we know that $p | a$. By item i), $p \nmid b$. Write 
\begin{align*}
d_{2(j+1)} \ =\ ad_{2(j+1)-1} + bd_{2(j+1)-2}&\ =\ ad_{2j + 1} + bd_{2j}\\
&\ =\ (a^2+b)d_{2j} + abd_{2j-1}.
\end{align*}
By the inductive hypothesis, $p | d_{2j}$. Since $p | a$, we obtain $p | d_{2(j+1)}$. Furthermore, write 
$$
d_{2(j+1)-1}\ =\ ad_{2(j+1)-2} + bd_{2(j+1)-3}\ =\ ad_{2j} + bd_{2j-1}.
$$
Since $p\nmid bd_{2j-1}$ by the inductive hypothesis and $p | ad_{2j}$, we obtain $p\nmid d_{2(j+1)-1}$. This completes our proof. 

iii) Suppose that $k = \gcd(b, d_i) > 1$ for some $i$. If $i = 2$, then $p|b$, which contradicts that $p | a$ and $\gcd(a, b) = 1$. If $i = 3$, then $q | b$, which contradicts $p^2 = aq + bp$. If $i\ge 4$, write $d_i = ad_{i-1} + bd_{i-2}$. Since $k|bd_{i-2}$ and $\gcd(a, b) = 1$, we get $k | d_{i-1}$ and so, $k | \gcd(b, d_{i-1})$. By induction, we obtain $k | \gcd(b, d_{3})$, which has been shown to be impossible.

Next, suppose that $k = \gcd(a, d_{2i-1}) > 1$ for some $i$. If $i = 2$, then $q | a$, contradicting $r = ap^2 + bq$. If $i\ge 3$, 
$$k\ |\ d_{2i-1}, \mbox{ }d_{2i-1} \ =\ ad_{2i-2} + bd_{2i-3}, \mbox{ and }\gcd(a, b) \ =\ 1\ \Longrightarrow\ k \ |\ d_{2i-3}.$$
By induction, $k | d_3$; that is, $\gcd(a, d_3) > 1$, which has been shown to be impossible. 

iv) The claim holds for $i\le 4$. For $i\ge 5$, using the recurrence $d_{i+1} = ad_i + bd_{i-1}$, we have 
$$\gcd(d_i, d_{i+1})\ =\ \gcd(d_i, bd_{i-1})\ \stackrel{\mbox{iii)}}{=}\ \gcd(d_i, d_{i-1}).$$
By induction, we obtain $\gcd(d_i, d_{i+1}) = 1$. 

v) The claim holds for $i \le 2$. For $i\ge 3$, we have
$$\gcd(d_{2i-1}, d_{2i+1})\ =\ \gcd(d_{2i-1}, ad_{2i})\ \stackrel{\mbox{iii)}}{=}\ \gcd(d_{2i-1}, d_{2i}) \ \stackrel{\mbox{iv)}}{=}\ 1.$$
This completes our proof. 
\end{proof}

\begin{lem}\label{l3}
Suppose that the first $4$ numbers in $S'_N$ are $p < q < r < p^2$. If $N$ is small recurrent with U(p, q, a, b), the following hold:
\begin{enumerate}
    \item[i)] $\gcd(a, b) = 1$ and $p\nmid a$.
    \item [ii)] For all $d_i\in S'_N$, $\gcd(b, d_i) = 1$.
    \item [iii)] For all $d_i, d_{i+1}\in S'_N$, $\gcd(d_i, d_{i+1}) = 1$.
    \item [iv)] Let $d_i\in S'_N$. Then $p|d_i$ if and only if $i\equiv 2\Mod 3$.
\end{enumerate}
\end{lem}
\begin{proof}
i) We have $\gcd(a, b)$ divides $r$ because $r = aq + bp$. Furthermore, $\gcd(a, b)$ divides $p$ because $p^2 = ar + bq$. Hence, $\gcd(a, b) = 1$.

Since $r= aq + bp$ and $r$ is a prime, $p\nmid a$.

ii) Suppose that $k = \gcd(b, d_i) > 1$ for some $i$. If $i = 2$, then $k = p$ and $p | b$. It follows from $p^2 = ar + bq$ that $p|a$, which contradicts that $\gcd(a, b) = 1$. If $i = 3$, then $k = q$ and $q | b$. It follows from $r = aq + bp$ that $q | r$, a contradiction. Hence, $i\ge 4$. Write 
$$1 \ <\ \gcd(b, d_i)\ =\ \gcd(b, ad_{i-1} + bd_{i-2})\ =\ \gcd(b, d_{i-1}).$$
By induction, we obtain $\gcd(b, d_3)>1$, which has been shown to be impossible. 

iii) The claim holds for $i\le 4$. Let $d_i, d_{i+1}\in S'_N$ and $i\ge 5$. Then
$$\gcd(d_i, d_{i+1}) \ =\ \gcd(d_i, ad_i + bd_{i-1})\ =\ \gcd(d_i, bd_{i-1})\ \stackrel{\mbox{ii)}}{=}\ \gcd(d_i, d_{i-1}).$$
By induction, $\gcd(d_i, d_{i+1}) = 1$. 

iv) The claim holds for $i\le 4$. Suppose that the claim holds for all $i \le j$ for some $j\ge 5$. We show that it also holds for $i = j+1$ under the assumption that $d_{j+1}\in S'_N$. We have
$$p^2 \ =\ ar + bq \ =\ a(aq+bp) + bq \ =\ (a^2+b)q + abp.$$
Hence, $p | (a^2 + b)$. Write
$$d_{j+1} \ =\ ad_j + bd_{j-1} \ =\ a(ad_{j-1} + bd_{j-2}) + bd_{j-1} \ =\ (a^2+b)d_{j-1} + abd_{j-2}.$$
Therefore, $p|d_{j+1}$ if and only if $p|d_{j-2}$. By the inductive hypothesis, we know that $p|d_{j+1}$ if and only if $j+1\equiv 2\Mod 3$.
\end{proof}

The proofs of the next two lemmas are similar to those of Lemmas \ref{l1} and \ref{l3}. Thus, we move their proofs to the Appendix.

\begin{lem}\label{l4}
Suppose that the first $4$ numbers in $S'_N$ are $p < q < r < pq$. If $N$ is small recurrent with $U(p, q, a, b)$, the following hold:
\begin{enumerate}
    \item [i)] $\gcd(a, b) = 1$ and $p\nmid a$.
    \item [ii)] For all $d_i\in S'_N$, $\gcd(b, d_i) = 1$.
    \item [iii)] For all $d_i, d_{i+1}\in S'_N$, $\gcd(d_i, d_{i+1}) = 1$.
    \item [iv)] Let $d_i\in S'_N$. Then $p|d_i$ if and only if $i\equiv 2\Mod 3$.
    \item [v)] Let $d_i\in S'_N$. Then $q|d_i$ if and only if $i$ is odd. 
\end{enumerate}
\end{lem}

\begin{lem}\label{l6}
Suppose that the first $4$ numbers in $S'_N$ are $p < q < r < s$. If $N$ is small recurrent with $U(p, q, a, b)$, the following hold:
\begin{enumerate}
    \item [i)] $\gcd(a, b) = 1$. 
    \item [ii)] For all $d_i\in S'_N$ with $i\ge 3$, $\gcd(b, d_i) = 1$.
    \item [iii)] For $d_i, d_{i+1}\in S'_N$, $\gcd(d_i, d_{i+1}) = 1$. 
    \item [iv] For $d_i\in S'_N$, $\gcd(a, d_i) = 1$.
    \item [v)] For $d_i, d_{i+2}\in S'_N$, $\gcd(d_i, d_{i+2}) = 1$. 
\end{enumerate}
\end{lem}

\section{Small recurrent numbers}\label{small}
We first find all small recurrent numbers with $|S'_N|\ge 4$. Then we check which $N$ is small recurrent out of all $N$ with $|S'_N|\le 3$ at the end of this section. As we rely heavily on case analysis, we underline possible forms of $N$ throughout our analysis for the ease of later summary. 

\subsection{The case $|S'_N|\ge 4$}  Let $d_2 = p$ for some prime $p$. Then $d_3$ is either $p^2$ or $q$ for some prime $q > p$. If $d_3 = p^2$, according to Proposition \ref{p1}, we know that \underline{$N = p^k$ or $N = p^kq$ for some $k\ge 1$ and $q > p^k$}. 

Assume, for the rest of this subsection, that $d_3 = q$ for some prime $q > p$. Then $d_4 = pq, p^2, r$ for some prime $r > q$.

\subsubsection{When $d_4 = pq$}
Since $S'_N$ satisfies $U(p, q, a, b)$, we get $pq = aq + bp$. So, $p|a$ and $q|b$. Write $a = pm$ for some $m\in \mathbb{Z}$ and get $b = (1-m)q$. Since $p^2\nmid N$ and $q$ divides $d_5 = apq + bq$, we can write 
$$d_5 \ =\ p^s q^tr_1^{\ell_1}\cdots {r_k}^{\ell_k},$$
where $s\le 1$, $t\ge 1$, and $r_i$'s are primes strictly greater than $pq$. If some $\ell_i\ge 1$, then $pq < r_i < d_5$ and $r_i\in S'_N$, a contradiction. Hence, $\ell_i = 0$ for all $i\le k$ and $d_5 = p^sq^t$. Since $d_5 > pq$ and $s\le 1$, we know that $t\ge 2$. 

\begin{enumerate}
    \item[a)] If $s = 1$, $d_5 = pq^2 > q^2 > d_4$ and $q^2|d_5$, so $q^2\in S'_N$, a contradiction. 
    \item[b)] If $s = 0$, $d_5 = q^t$ for some $t\ge 2$. Since $d_5$ is the next number after $d_4$ in increasing order, $d_5 = q^2$. Using the linear recurrence, we obtain $q^2 = apq + bq$, so $q = ap + b = p^2m + (1-m)q$. It follows that $p^2m = mq$. We arrive at $m = 0$, $a = 0$, and $b = q$. Hence, all elements of $S'_N$ are divisible by either $p$ or $q$. If $N$ has a prime factor $r\ge \sqrt{N}$, by Proposition \ref{p2}, $r$ is unique. We conclude that \underline{$N = pq^k$ or $pq^kr$ for some $k\ge 2$, $p < q$, and $pq^k < r$}.
\end{enumerate}

\subsubsection{When $d_4 = p^2$}
The first few divisors of $N$ are $1 < p < q < p^2$. 
We have $p^2 = aq + bp$, so $p|a$. Write $a = pm$ and get $b = p-mq$. We argue for possible forms of $d_5$. Let $r$ be the largest prime factor of $d_5$. If $r > q$, then $r > p^2$ and $d_5 = r$. Otherwise, if $r \le q$, then $d_5 = p^\ell q^k$ for some $\ell, k\ge 0$. Suppose that $k\ge 2$. We get $$d_5\ \ge\ q^2 \ >\ pq \ >\ d_4\mbox{ and }pq\  |\ N,$$ a contradiction. Hence, $k\le 1$. If $k = 0$, then 
$$d_5 \ =\ p^3 \ >\ pq \ >\ d_4\mbox{ and }pq \ |\ N,$$
another contradiction. Therefore, $k = 1$ and $d_5 = pq$. We conclude that either $d_5 = r$ for some $r > q$ or $d_5 = pq$. 

\begin{enumerate}
    \item [a)] If $d_5 = pq$, then $bq + ap^2 = pq$. So, $(p-mq)q + mp^3 = pq$, which gives $mq^2 = mp^3$. Hence, $m = 0$, $a = 0$, and $b = p$. We know that elements of $S'_N$ are divisible by either $p$ or $q$. If $N$ has a prime factor $r'$ at least $ \sqrt{N}$, by Proposition \ref{p2}, $r'$ is unique. Hence, either $N = p^\ell q^k$ or $p^\ell q^k r'$ for some prime $r' > p^\ell q^k$, $\ell \ge 2$, and $k\ge 1$.
    
    Case a.i) $N = p^\ell q^k r'$. We claim that $k = 1$. Indeed, if $k\ge 2$, then 
    $$q^4 \ <\ q^2 r' \ <\ N \ \Longrightarrow\ q^2 \ <\ \sqrt{N}\ \Longrightarrow\ q^2\in S'_N.$$
    Since $b = p$, we know that $p|d$ for all $d \ge d_4$ and $d\in S'_N$, which contradicts $q^2\in S'_N$. Hence, \underline{$N = p^\ell q r'$ for some $\ell\ge 2$, some prime $r' > p^\ell q$, and $\sqrt{q} < p < q$.}
    
    Case a.ii) $N = p^\ell q^k$. As above, $q^2\notin S'_N$. If $k\ge 2$, then 
    $$q^2\ \ge\ \sqrt{N} \ \Longrightarrow\ q^4 \ \ge\ N \ =\  p^{\ell}q^{k}\ =\ p^{\ell-2}p^2q^k\ >\ q^{k+1}.$$
    Hence, $k < 3$, which implies that $k = 2$. In this case, $N = p^\ell q^2$ and 
    $$q^2 \ >\ p^\ell \ >\ q^{\ell/2}\ \Longrightarrow \ \ell \ \le 3.$$
    We conclude that one of the following holds:
    \begin{itemize}
        \item \underline{$N = p^\ell q$ for some $\ell\ge 2$ and $\sqrt{q} < p < q$},
        \item \underline{$N = p^2 q^2$ for some $p < q < p^2$},
        \item \underline{$N = p^3 q^2$ for some $p^{3/2} < q < p^2$}.
    \end{itemize}
    \item[b)] $d_5 = r$. 

\begin{prop}\label{l5}
Suppose that the first $4$ numbers in $S'_N$ are $p < q < p^2 < r$. Then $|S_N| \le 7$. As a result, $|S'_N| \le 6$.
\end{prop}

\begin{proof}
Assume that $|S_N|\ge 2i$ for some $i\ge 4$. We obtain a contradiction by showing that $|S_N|\ge 2i+2$. By Lemma \ref{l1} item ii), $p\nmid d_{2i-1}$, $p|d_{2i-2}$, and $p\nmid d_{2i-3}$. By Lemma \ref{l1} item v), $\gcd(d_{2i-1}, d_{2i-3}) = 1$, so $p^2d_{2i-1}d_{2i-3}$ divides $N$. Hence, $pd_{2i-3}\in S_N'$. 

If $pd_{2i-3} = d_{2i-2}$, then 
$$pd_{2i-3} \ =\ ad_{2i-3} + bd_{2i-4} \ \Longrightarrow \ d_{2i-3}\ |\ bd_{2i-4},$$
which contradicts Lemma \ref{l1} items iii) and iv).

If $pd_{2i-3} = d_{2i}$, then
\begin{align*}pd_{2i-3} \ =\ ad_{2i-1} + bd_{2i-2} &\ =\ a(ad_{2i-2} + bd_{2i-3}) + bd_{2i-2}\\
&\ =\ (a^2+b)d_{2i-2} + abd_{2i-3}.
\end{align*}
Therefore, $d_{2i-3}$ divides $a^2+b$. It is easy to check that for $d_j\in S'_N$, the sequence $d_j\Mod a^2+b$ is congruent to 
$$1, p, q, p^2, abp, abq, abp^2, (ab)^2p, (ab)^2q, (ab)^2p^2, \ldots.$$
Hence, we can write
$$d_{2i-3} \ =\ (a^2+b)\ell + a^kb^ks,$$
for some $\ell\in \mathbb{Z}$, some $k\ge 1$, and some $s\in \{p, q, p^2\}$. Since $d_{2i-3}|(a^2+b)$, $d_{2i-3}|a^kb^ks$. By Lemma \ref{l1} item iii), $d_{2i-3}|s$; that is, $d_{2i-3} \le p^2$. However, $d_{2i-3}\ge d_5 > d_4 = p^2$, a contradiction. 

We conclude that $pd_{2i-3} \ge d_{2i+2}$. Since $pd_{2i-3}\in S'_N$, we know that $d_{2i+2}\in S'_N$ and $|S_N|\ge 2i+2$.
\end{proof}

\begin{prop}\label{l2}
Suppose that the first $4$ numbers in $S'_N$ are $p < q < p^2 < r$. If $N$ is small recurrent, then $|S'_N| \neq 4, 6$.
\end{prop}

\begin{proof}
If $|S'_N| = 4$, then \eqref{e5} gives $\tau(N) = 10$ or $11$. Note that $N$ has three distinct prime factors $p, q, r$ and the power of $p$ is at least $2$. Since $2^3\cdot 3 > 11$, $N$ cannot have another prime factor besides $p, q, r$. Write $N = p^a q^b r^c$, for some $a\ge 2, b\ge 1, c\ge 1$. However, neither $(a+1)(b+1)(c+1) = 10$ nor $(a+1)(b+1)(c+1) = 11$ has a solution. Therefore, $|S'_N| \neq 4$. A similar argument gives $|S'_N|\neq 6$.
\end{proof}

By Propositions \ref{l5} and \ref{l2}, we know that $|S'_N| = 5$; that is, $\tau(N) = 12$ or $13$. Using the same reasoning as in the proof of Proposition \ref{l2}, we know that $\tau(N) = 12$ and \underline{$N = p^2qr$, where $p < q < p^2 < r$}.
\end{enumerate}

\subsubsection{When $d_4 = r$ for some $r > q$}
The possible values for $d_5$ are $p^2, pq, s$ for some prime $s > r$. 

\begin{enumerate}
    \item[a)] If $d_5 = p^2$, we can generalize the method by Chentouf.
    \begin{prop}\label{p4}
    If $N$ is small recurrent and the first four numbers of $S'_N$ are $p < q< r < p^2$, then $|S_N|\le 7$. As a result, $|S'_N|\le 6$.
    \end{prop}
    \begin{proof}
    Suppose that $|S_N|\ge 8$.
    We show that $|S_N|\ge 3i+2$ for all $i\in \mathbb{N}$, which is a contradiction. The claim holds for $i= 2$. Assume that $|S_N|\ge 3j+2$ for some $j\ge 2$. By Lemma \ref{l3}, $p\nmid d_{3j}d_{3j+1}$ and $\gcd(d_{3j}, d_{3j+1}) = 1$. Hence, $p^2d_{3j}d_{3j+1}$ divides $N$, which implies that $pd_{3j}\in S'_N$.
    
    If $pd_{3j} = d_{3j+2} = ad_{3j+1} + bd_{3j}$, then $d_{3j}$ divides $ad_{3j+1}$. By Lemma \ref{l3}, $d_{3j}|a$. Observe that for $d_i\in S'_N$, the sequence $d_i\Mod a$ is $$1, p, q, bp, bq, b^2p, b^2q, \ldots.$$
    Write $d_{3j} = a\ell + b^ks$, for some $\ell\in\mathbb{Z}$, some $k\ge 0$, and some $s\in \{p, q\}$. We see that $d_{3j}|b^ks$ for some $k\ge 0$ and $s\in \{p, q\}$. By Lemma \ref{l3}, $d_{3j} \le q$. However, 
    $$d_{3j} \ \ge\ d_6 \ >\  d_{3} \ =\ q,$$
    a contradiction.
    
    If $pd_{3j} > d_{3j+2}$, then $pd_{3j} \ge d_{3(j+1)+2}$ by Lemma \ref{l3}. Therefore, $|S_N|\ge 3(j+1)+2$.
    \end{proof}
    
    \begin{prop}\label{p10}
    There is no small recurrent $N$ whose the first four numbers of $S'_N$ are $p < q< r < p^2$.
    \end{prop}
    \begin{proof}
    By Proposition \ref{p4}, $|S'_N| \in \{4, 5, 6\}$. If $|S'_N| = 4$, then $\tau(N) = 10$ or $11$, none of which can be written as a product of at least three integers, each of which is at least $2$. This contradicts \eqref{e5} and the fact that $N$ has three distinct prime factors.  We arrive at the same conclusion when $|S'_N| = 6$. For $|S'_N| = 5$, we obtain $N = p^2qr$ for some primes $p < q < r < p^2$. However, this poses another contradiction. Observe that $(pq)^2 < p^2qr$, so the divisors in $S'_N$ are $p < q < r < p^2 < pq$. Since $pq = ap^2 + br$, we get $p|b$, which contradicts Lemma \ref{l3} item ii). 
    \end{proof}
    \item[b)] Suppose that $d_5 = pq$.
    \begin{prop}
    There is no small recurrent number $N$ such that the first four numbers of $S'_N$ are $p < q< r < pq$.
    \end{prop}
    \begin{proof}
     Assume that $|S'_N|\ge 8$. 
     Since $p^2\notin S'_N$, $p$ divides $N$ exactly. 
     By Lemma \ref{l4}, $d_6$ is divisible neither by $p$ nor $q$. Hence, $d_6 = s$ for some prime $s > pq$. The divisor $d_7$ is divisble by $q$; hence, $d_7 = qr$ or $q^2$. The divisor $d_8$ is divisible by $p$ but not by $q$. 
     So, $d_8 = pr$, which gives that $d_7$ must be $q^2$ because $d_7 < d_8$.
     Now $q|d_9\mbox{ and } p\nmid d_9\Longrightarrow d_9= qr$. However, that $\gcd(d_8, d_9) = r$ contradicts Lemma \ref{l4}. Therefore, $|S'_N|\in\{4, 5, 6, 7\}$. Using the same argument as in the proof of Proposition \ref{p10}, we know that $|S'_N|\neq 4, 6$ and so, 
     $|S'_N|\in \{5, 7\}$. By the above argument, if $|S'_N|\ge 5$, then $d_6$ is a prime greater than $pq$. Hence, $N$ has at least $4$ distinct prime factors, so $\tau(N)$ can be written as a product of at least $4$ integers greater than $1$. Clearly, \eqref{e1} rules out the case $|S'_N| = 5$. If $|S'_N| = 7$, the above argument shows that $q^2|N$; hence, $\tau(N)$ can be written as a product of at least $4$ integers greater than $1$, one of which is greater than $2$.
     This cannot happen as $\tau(N)\in \{16, 17\}$. 
    \end{proof}
    \item[c)] Suppose that $d_5 = s$. 
    \begin{prop}
    There is no small recurrent number $N$ such that the first four numbers of $S'_N$ are $p < q< r < s$.
    \end{prop}
    \begin{proof}
    Observe that $pq$ and $pr$ are in $S'_N$. Let $d_j = pv$ be the largest element of $S'_N$ that is divisible by $p$. Clearly, $v > p$ and $j\ge 7$. By Lemma \ref{l6}, $d_{j}, d_{j-1}$, and $d_{j-2}$ are pairwise coprime. Hence, $pvd_{j-1}d_{j-2}$ divides $N$, so $pd_{j-2}\in S'_N$. 
    
    If $pd_{j-2} = d_{j-1}$, then $p|d_{j-1}$ and so, $p|\gcd(d_{j-1}, d_j)$, which contradicts Lemma \ref{l6} item iii).
    
    If $pd_{j-2} = d_j$, then $d_{j-2} = v$ and $\gcd(d_{j-2}, d_j) = v > 1$, which contradicts Lemma \ref{l6} item v).
    
    Therefore, we have $pd_{j-2} > d_j$, which, however, contradicts that $d_j$ is the largest element of $S'_N$ that is divisible by $p$. We conclude that there is no small recurrent number $N$ such that the first four numbers of $S'_N$ are $p < q< r < s$.
    \end{proof}
\end{enumerate}

From the above analysis, we arrive at the following proposition.

\begin{prop}\label{p20}
If $N$ is small recurrent and $|S'_N|\ge 4$, then $N$ belongs to one of the following forms.
\begin{enumerate}
    \item[(S1)] $N = p^k$ or $N = p^kq$ for some $k\ge 1$ and $q > p^k$.
    \item[(S2)] $N = pq^k$ or $pq^kr$ for some $k\ge 2$, $p < q$, and $pq^k < r$.
    \item[(S3)] $N = p^k q r$ for some $k\ge 2$, some prime $r > p^k q$, and $\sqrt{q} < p < q$.
    \item[(S4)] $N = p^k q$ for some $k\ge 2$ and $\sqrt{q} < p < q$.
    \item[(S5)] $N = p^2 q^2$ for some $p < q < p^2$.
    \item[(S6)] $N = p^3 q^2$ for some $p^{3/2} < q < p^2$.
    \item[(S7)] $N = p^2qr$, where the first four numbers in $S'_N$ are $p < q < p^2 < r$.
\end{enumerate}
\end{prop}

These forms together establish the necessary condition for a small recurrent $N$ to have $|S'_N|\ge 4$.  
We now refine each form (if necessary) to obtain a necessary and sufficient condition. 

\begin{enumerate}
    \item[(S1)]
    \begin{itemize}
        \item If $N = p^k$, then $N$ is small recurrent with $|S'_N|\ge 4$ if $k\ge 9$. In this case, $S'_N = \{p, p^2, p^3, \ldots, p^{\lfloor (k-1)/2\rfloor}\}$ satisfies $U(p, p^2, p, 0)$. 
        \item If $N = p^k q$ for some $k\ge 1$ and $q > p^k$, then $N$ is small recurrent with $|S'_N|\ge 4$ if $k\ge 4$. In this case, $S'_N = \{p, p^2, p^3, \ldots, p^k\}$ satisfies $U(p, p^2, p, 0)$.
    \end{itemize}

    \item[(S2)]
    \begin{itemize}
        \item If $N = pq^k$ for some $k\ge 2$ and $p < q$, then $N$ is small recurrent with $|S'_N|\ge 4$ if  $\sqrt{N} = \sqrt{pq^k} > q^2$. Hence, $N = pq^k$ for some $k\ge 4$ and $p < q$.
        In this case, $S'_N = \{p, q, pq, q^2, \ldots\}$ satisfies $U(p, q, 0, q)$.
        \item If $N = pq^kr$ for some $k\ge 2$, $p < q$, and $r > pq^k$, then $N$ is small recurrent with $|S'_N|\ge 4$.
        In this case, $S'_N = \{p, q, pq, q^2, \ldots, pq^{k-1}, q^k, pq^k\}$ satisfies $U(p, q, 0, q)$.
    \end{itemize}
    
    \item[(S3)]
    If $N$ belongs to (S3), then $S'_N = \{p, q, p^2, pq, \ldots, p^k, p^{k-1}q, p^k q\}$. Since $p^2 = aq + bp$, we know that $p|a$. Write $a = pm$ for some $m\in\mathbb{Z}$ and get $b = p-mq$. Hence,
    $$pq \ =\ ap^2 + bq\ =\ p^3m + (p-mq)q\ \Longrightarrow\ mq^2 = p^3m.$$
    Therefore, $(m, a, b) = (0, 0, p)$. However, the largest element in $S'_N$, $p^k q$, is not equal to  $p\cdot p^k$. We conclude that form (S3) does not give a small recurrent number. 
    
    \item[(S4)] If $N = p^k q$ for some $k\ge 2$ and $\sqrt{q} < p < q$, then the nontrivial divisors of $N$ in increasing order is $p < q < p^2 < pq < \cdots$. In order that $|S'_N|\ge 4$, we need $(pq)^2 < p^kq$, so $q < p^{k - 2}$. Hence, $k\ge 4$. In this case, $S'_N = \{p, q, p^2, pq, \ldots\}$ satisfies $U(p, q, 0, p)$. 
    
    \item[(S5)] If $N = p^2 q^2$ for some $p < q < p^2$, then $\tau(N) = 9$. However, if $|S'_N|\ge 4$, then $\tau(N)\ge 10$ by \eqref{e5}. We conclude that form (S5) does not give a small recurrent number. 
    
    \item [(S6)] If $N = p^3 q^2$ for some $p^{3/2} < q < p^2$, then $S'_N = \{p, q, p^2, pq, p^3\}$ satisfies $U(p, q, 0, p)$.
    
    \item [(S7)] Let $N$ have form (S7). Since the first four numbers of $S'_N$ are $p < q < p^2 < r$ and $\tau(N) = 12$, we know that the fifth number in $S'_N$ must be $pq$. 
    That $p < q< p^2 < r < pq$ satisfies some $U(p, q , a, b)$ gives $a = \frac{p(pq - r)}{q^2 - p^3}$, $b = \frac{rq - p^4}{q^2 - p^3}$, and $r = pq-\sqrt{(q^2-p^3)(p^2-q)}$. We conclude that a number of form (S7) is small recurrent if and only if $p < q < p^2 < r < pq$, $(q^2 - p^3)|(pq - r)$, $(q^2-p^3)|(rq-p^4)$, and $r =  pq-\sqrt{(q^2-p^3)(p^2-q)}$. An example is $(p, q, r) = (2, 3, 5)$. We do not know if $(2, 3, 5)$ is the only set of primes that satisfy all these conditions or not. 
\end{enumerate}

From the above analysis, we obtain the proposition, which is a refinement of Proposition \ref{p20}.
\begin{prop}\label{p21} Let $p, q, r$ denote prime numbers and $k$ be some positive integer. 
A positive integer $N$ is small recurrent with $|S'_N|\ge 4$ if and only if $N$ belongs to one of the following forms.
\begin{enumerate}
    \item $N = p^k$ for some $k\ge 9$. In this case, $S'_N = \{p, p^2, p^3, \ldots, p^{\lfloor (k-1)/2\rfloor}\}$ satisfies $U(p, p^2, p, 0)$. 
    \item $N = p^k q$ for some $k\ge 4$ and $q > p^k$. In this case, $S'_N = \{p, p^2, p^3, \ldots, p^k\}$ satisfies $U(p, p^2, p, 0)$.
    \item $N = pq^k$ for some $k\ge 4$ and $p < q$. In this case, $S'_N = \{p, q, pq, q^2, \ldots\}$ satisfies $U(p, q, 0, q)$.
    \item  $N = pq^kr$ for some $k\ge 2$, $p < q$, and $r > pq^k$.
        In this case, $S'_N = \{p, q, pq, q^2, \ldots, pq^{k-1}, q^k, pq^k\}$ satisfies $U(p, q, 0, q)$.
    \item  $N = p^k q$ for some $k\ge 4$ and $\sqrt{q} < p < q$. In this case, $S'_N = \{p, q, p^2, pq, \ldots\}$ satisfies $U(p, q, 0, p)$. 
    \item $N = p^3 q^2$ for some $p^{3/2} < q < p^2$. In this case, $S'_N = \{p, q, p^2, pq, p^3\}$ satisfies $U(p, q, 0, p)$.
    \item $N = p^2qr$, where $p < q < p^2 < r < pq$, $(q^2 - p^3)|(pq - r)$, $(q^2-p^3)|(rq-p^4)$, and $r =  pq-\sqrt{(q^2-p^3)(p^2-q)}$. In this case, $S'_N = \{p, q, p^2, r, pq\}$ satisfies $U\left(p, q, \frac{p(pq - r)}{q^2 - p^3}, \frac{rq - p^4}{q^2 - p^3}\right)$.
\end{enumerate}
\end{prop}

\subsection{The case $|S'_N|\le 3$}

If $|S'_N|\le 3$, then $\tau(N)\le 9$. We use the classifications of those $N$ from the introduction to obtain the following proposition.
\begin{prop}\label{p22}
Let $p, q, r$ denote prime numbers and $k$ be some positive integer. 
A positive integer $N > 1$ is small recurrent with $|S'_N|\le 3$ if and only if $N$ belongs to one of the following forms.
\begin{enumerate}
    \item $N = p^k$ for some $k\le 8$. In this case, $S'_N = \{p, p^2, \ldots, p^{\lfloor (k-1)/2\rfloor}\}$ satisfies $U(p, p^2, p, 0)$.
    \item $N = pq$ for some $p < q$. In this case, $S'_N = \{p\}$.
    \item $N = pq^2$ for some $p < q$. In this case, $S'_N = \{p, q\}$.
    \item $N = p^2q$ for some $p < q$. If $q < p^2$, then $S'_N = \{p, q\}$. If $q > p^2$, then $S'_N = \{p, p^2\}$.
    \item $N = pq^3$ for some $p < q$. In this case, $S'_N = \{p, q, pq\}$.
    \item \label{i1}$N = p^3q$ for some $p < q$. If $p < q < p^2$, then $S'_N = \{p, q, p^2\}$. If $p^3 < q$, then $S'_N = \{p, p^2, p^3\}$. (The case $p^2 < q < p^3$ is eliminated because the three elements in $S'_N$ would be $p < p^2 < q$. However, there is no integral solution $(a, b)$ to $q = ap^2 + bp$.) 
    \item $N = p^2 q^2$ for some $p < q$. If $p < q < p^2$, then $S'_N = \{p, q, p^2\}$. (The case $p^2 < q$ is eliminated due to the same reason as in item \eqref{i1}.)
    \item $N = pqr$ for some $p < q < r$. If $r < pq$ and there is an integral solution $(a, b)$ to $r = aq + bp$, then $S'_N = \{p, q, r\}$. If $r > pq$, then $S'_N = \{p, q, pq\}$.
\end{enumerate}
\end{prop}

Combining Propositions \ref{p21} and \ref{p22}, we obtain Theorem \ref{m1}.

\section{Large recurrent numbers}\label{large}

Now we characterize all positive integers $N$ whose $L'_N$ satisfies a linear recurrence of order at most two. By a simple observation, instead of working directly with divisors in $L'_N$, we work with divisors in $S'_N$. Again, the set of divisors of a positive integer $N$ is $1 = d_1 < d_2 <\cdots < d_{\tau(N)}$ and the set $S'_N = \{d_2, d_3, \ldots\}$.

\subsection{The case $|L_N'|\ge 4$}

Note that $|L_N'|\ge 4$ is equivalent to $|S_N'|\ge 4$.

\begin{lem}\label{p222}
    For any $d\in L'_{N}$, we have $N/d\in S'_{N}$. If $N$ is large recurrent with $|L_N'|\ge 4$, then \begin{equation}\label{ej1} ad_{i+2}+bd_{i+1}\ =\ \frac{d_{i+1}d_{i+2}}{d_i}, \forall d_i, d_{i+1}, d_{i+2}\in S'_{N}.\end{equation} In particular, we have 
    \begin{equation}\label{ej2} ad_4+bd_3 \ =\ \frac{d_3d_4}{d_2}.\end{equation}
    \end{lem}
    \begin{proof}
    If $d\in L'_{N}$, then $\sqrt{N} < d < N$. Then $1 < N/d < \sqrt{N}$, so $N/d\in S'_{N}$. 
    Let 
    $$d'_i \ :=\ d_{\tau(n)+1-i}\ =\ \frac{N}{d_i}\in L'_N, \forall d_i\in S'_N.$$ 
    If $N$ is large recurrent, then we have 
    $$d'_{i}\ =\ ad_{i+1}^{'}+bd_{i+2}^{'}, \forall d'_i, d'_{i+1}, d'_{i+2}\in L^{'}_N.$$
    Therefore,
    $$\frac{N}{d_i}\ =\ a\frac{N}{d_{i+1}}+b\frac{N}{d_{i+2}},\forall d_i, d_{i+1}, d_{i+2}\in S'_N,$$
    which gives
    $$ad_{i+2}+bd_{i+1}\ =\ \frac{d_{i+1}d_{i+2}}{d_i}, \forall d_i, d_{i+1}, d_{i+2}\in S'_N.$$
    This completes our proof.
    \end{proof}

Since $d_2$ is a prime number $p$ and $d_3$ is either $p^2$ or a  prime number $q>p$, we consider two cases.
\subsubsection{When $d_3=p^2$}
Then $d_4$ is either $p^3$ or a prime number $q>p^2$. 

\begin{enumerate}

\item[a)] 
 If $d_4=p^3$, then \eqref{ej2} implies that $p^2 = ap + b$.
 \begin{claim}
If $p\neq a$, then $S'_N=\{p,p^2,\ldots,p^k\}$ for some $k\ge 4$.
 \end{claim}
 \begin{proof}
We need to show that if $d_i\in S'_N$, then $d_i = p^{i-1}$.  Base case: the claim holds for $i\le 4$. Suppose that there exists a $j\ge 4$ such that $d_{i} = p^{i-1}$ for all $i\le j$. Using \eqref{ej1}, we have
 $$ad_{j+1} + bp^{j-1} \ =\ ad_{j+1} + bd_j \ =\ \frac{d_{j+1}d_j}{d_{j-1}}\ =\ \frac{d_{j+1} p^{j-1}}{p^{j-2}}\ =\ pd_{j+1},$$
 which, combined with $p^2 = ap + b$, gives 
 $$(p-a)(d_{j+1}-p^j) \ =\ 0.$$
 Since $p\neq a$, we obtain $d_{j+1} = p^j$, as desired. 
 \end{proof}
 By Proposition \ref{p1}, we know that when $p \neq a$, \underline{either $N = p^k$ or $N = p^kq$ for} \underline{some $k\ge 1$ and some prime $q > p^k$.}

Now suppose that $p = a$.  Then $b = 0$. We can write elements in $L'_N$ as $\{g_1, g_2, pg_2, p^2g_2, \ldots, p^kg_2\}$ for some $k\ge 2$. Correspondingly, the set $S'_N$ is $\{p, p^2, \ldots, p^{k}, p^{k+1}, p^{k+1}g_2/g_1\}$. If $p^{k+1}g_2/g_1$ is a power of $p$, then we have the same conclusion about $N$ as when $p\neq a$. If $p^{k+1}g_2/g_1$ is not a power of $p$, then 
$$\frac{p^{k+1}g_2}{g_1} \ =\ q, \mbox{ for some prime }q > p\ \Longrightarrow\ g_1 \ =\ p^{k+1}\frac{g_2}{q}.$$
Note that $g_2/q\in \mathbb{N}$. Furthermore, we claim that $g_2/q = p$. Indeed, since $1< g_2/q < g_1$, we know that $g_2/q\in S'_N$. If $g_2/q = q$, then 
$$pq \ <\ p^{k+1}\frac{g_2}{q}\ =\ g_1,$$
which implies that $pq\in S'_N$, a contradiction. If $g_2/q = p^j$ for some $j > 1$, then 
$$p^{k+2}\ <\ p^{k+1+j}\ =\ p^{k+1}\frac{g_2}{q} \ =\ g_1,$$
which implies that $p^{k+2}\in S'_N$, another contradiction. 
Therefore, $g_2/q = p$, and we obtain $g_1 = p^{k+2}$ and $g_2 = pq$. We conclude that \underline{$N = p^{k+2}q$ for some $k\ge 2$} \underline{and $p^{k+1} < q < p^{k+2}$.}

\item [b)] If $d_4=q$, we claim that $a \neq p$. Suppose otherwise. Applying \eqref{ej2} to $d_2, d_3$, and $d_4$ gives $aq + bp^2 = pq$. Hence, $a = p$ implies that $b = 0$. However, applying \eqref{ej1} to $d_3, d_4$, and $d_5$ gives 
$(p^3-q)d_5 = 0$, a contradiction. Therefore, $a\neq p$.
By \eqref{ej2}, we have 
$$d_4\ =\ \frac{bp^2}{p-a}\ =\ q \ \Longrightarrow\ q|b\ \Longrightarrow\ b = kq\mbox{ for some }k\in \mathbb{Z}\backslash\{0\}.$$
Hence, $a=p-kp^2$. By \eqref{ej1} applied to $d_3, d_4$, and $d_5$, 
\begin{equation}\label{eq1}d_5\ =\ \frac{kp^2q^2}{q-p^3+kp^4},\end{equation}
which implies that $p^2|d_5$ since $\gcd(p^2,q-p^3+kp^4)=1$. Hence, $d_5=p^3$, and $\eqref{eq1}$ gives $k={p(p^3-q)\over p^5-q^2}$. 

\begin{itemize}
\item[] Case b.i) If $S'_N$ has exactly four elements, which are $p,p^2,q,p^3$, then $\tau(N) = 10$, which implies that $N = p^4q$. Hence, $L_N' = \{pq,p^4,p^2q,p^3q\}$ with $a=pq{p^2-q\over p^5-q^2}$ and $b=pq{p^3-q\over p^5-q^2}$. We conclude that \underline{$N=p^4q$ with $p^2 < q < p^3$, } \underline{$(p^5-q^2)|(p^2-q)$, and $(p^5-q^2)|(p^3-q)$}.

\item[] Case b.ii) If $|S'_N| > 4$, then \eqref{ej1} gives $$d_6\ =\ {bp^3q\over p^3-aq}\ \Longrightarrow\  q|d_6 \ \Longrightarrow\ d_6 \ =\ pq.$$
However, since $(a, b) = (p(1-kp), kq)$, we have
$$pq \ =\ d_6\ =\ {bp^3q\over p^3-aq}\ =\ {kp^2q^2\over p^2-q(1-kp)},$$ which gives $p^2 = q$, a contradiction.
\end{itemize}
\end{enumerate}

We summarize our result when $d_3 = p^2$. 

\begin{prop}\label{p23} A number $N$ is large recurrent with $|L'_N|\ge 4$ and $(d_2, d_3) = (p, p^2)$ for some prime $p$ if and only if $N$ belongs to one of the following forms.
\begin{enumerate}
    \item $N=p^k$ for some $k\ge 9$. In this case, $L'_N=\{p^{\lceil(k-1)/2\rceil+1}, p^{\lceil (k-1)/2\rceil+2},\ldots,p^{k-1}\}$ satisfies $U(p^{\lceil (k-1)/2\rceil+1},p^{\lceil (k-1)/2\rceil +2},p,0)$.
    \item $N=p^kq$ for some $k\ge 4$ and $q>p^k$. In this case, $L'_N=\{q,pq,p^2q,\ldots,p^{k-1}q\}$ satisfies $U(q,pq,p,0)$.
    \item $N=p^kq$ for some $k\ge 4$ and $p^{k-1} < q < p^k$. Then $$L'_N \ =\ \{p^k, pq, p^2q, \ldots, p^{k-1}q\}$$
    satisfies $U(p^k, pq, p, 0)$.
    \item $N=p^4q$ with $p^2 < q < p^3$, $(p^5-q^2)|(p^2-q)$, and $(p^5-q^2)|(p^3-q)$. In this case, 
    $L'_N \ =\ \{pq, p^4, p^2q, p^3q\}$.
\end{enumerate}
\end{prop}

\subsubsection{When $d_3=q$} 
By \eqref{ej1}, $$p(ad_4+bq)\ =\ d_4q\ \Longrightarrow\ p|d_4.$$
Write $d_4=kp$ for some integer $k$. Since $d_2 = p$ and $d_3 = q$, $d_4$ must be either $p^2$ or $pq$.

\begin{enumerate}

\item[a)] If $d_4=p^2$, \eqref{ej1} gives $ap^2 + bq = pq$. Hence, $q|a$ and $p|b$. Write $a=mq$ and $b=np$ for some integers $m,n$ to get $mp + n = 1$. By \eqref{ej1}, we see that 
$$d_5\ =\ \frac{bp^2q}{p^2-aq}\ \Longrightarrow\ q | d_5\ \Longrightarrow \ d_5 = pq.$$Therefore, $$\frac{bp^2q}{p^2-aq}\ =\ \frac{np^3q}{p^2-mq^2}\ =\ \frac{(1-mp)p^3q}{p^2-mq^2}\ =\ pq\ \Longrightarrow\ m(p^3-q^2)=0,$$
which gives $m = 0$ and so, $(a, b) = (0, p)$. By \eqref{ej1}, $d_{i+2} = pd_{i}$ for all $d_{i}, d_{i+2}\in S'_N$ and $$S'_N\ =\ \{p,q,p^2,pq, \ldots\}.$$ 
If $|S'_N| = 4$, then $\tau(N) = 10$ and \underline{$N = p^4q$ for $p < q < p^2$}. Suppose that $|S'_N| \ge 5$, then $p^3\in S'_N$. 
\begin{itemize}
    \item[] Case a.i) If $q^2|N$, let $k\ge 2$ and $\ell\ge 3$ be the largest power such that $q^k|N$ and $p^k|N$, respectively. Since $q^2\notin S'_N$, we know that $$q^4\ \ge\ N \ \ge\ p^3q^k \ >\ q^{k+3/2}\Longrightarrow\ k\ < \ 5/2.$$
It follows that $k = 2$. That $q^2 < p^2q$ implies that 
$$(p^2q)^2 \ >\  N\ \ge\ p^\ell q^2\ \Longrightarrow\ 3\ \ge\ \ell\ \ge\ 3.$$
Hence, $\ell=3$. If $N$ does not have any other prime divisors besides $p$ and $q$, then \underline{$N = p^3q^2$ for $p < q < p^2$}. If $N$ has a prime divisor $r\neq p, q$, then $r > \sqrt{N}$. So, $r$ must be the unique prime divisor different from $p$ and $q$. We have $N = p^3q^2r$ for $p < q< p^2$ and $r>p^3q^2$. Then $q^2\in S'_N$, a contradiction. 
    \item[] Case a.ii) If $q^2\nmid N$ and $N$ has no prime divisors other than $p$ and $q$, then \underline{$N = p^kq$ some for $k\ge 2$ and $p<q<p^2$}.
    \item[] Case a.iii) If $q^2\nmid N$ and there exists a prime divisor $r$ other than $p$ or $q$, then $r > \sqrt{N}$ and $r$ is the unique prime different from $p$ and $q$. Therefore, $N = p^kqr$ for some $k\ge 2$ and $p<q<p^2<p^kq<r$. Note that the two largest elements in $S'_N$ are $p^{k-1}q$ and $p^kq$. Let $d$ be the third largest divisor in $S'_N$. The relation $d_{i+2} = pd_{i}$ for all $d_{i}, d_{i+2}\in S'_N$ gives that $dp = p^kq$ and so, $d = p^{k-1}q$, which contradicts that $p^{k-1}q$ is the second largest in $S'_N$.
\end{itemize}

\item[b)] If $d_4=pq$, then $p^2\nmid N$ since $p^2<pq$. By \eqref{ej2}, 
\begin{equation}\label{eqd4} ap\ =\ q-b.\end{equation}
We see that $d_5$ is equal to $q^2$ or $r$, for some prime $r > pq$. 

\begin{itemize}
\item[] Case b.i) If $d_5=q^2$, then \eqref{ej1} gives 
\begin{equation}\label{eqd5}bp \ =\  (p-a)q.\end{equation}
From \eqref{eqd4} and \eqref{eqd5}, we obtain $a(p^2-q) = 0$, so $(a, b)= (0, q)$. By \eqref{ej1}, $d_{i+2} = qd_i$ for all $d_{i}, d_{i+2}\in S'_N$. Using \cite[Proposition 5]{C}, we conclude that \underline{$N=pq^k$ or $N=pq^kr$ for some $k\ge 2$ and $p < q < pq^k < r$}.

\item[] Case b.ii) If $d_5=r$, then we claim that $|S'_N| > 4$. If not, $|S'_N| = 4$ implies that $\tau(N) = 10$, which contradicts that $N$ has three distinct prime divisors. By \eqref{ej1}, we see that 
    $$pq(ad_6+br)\ =\ d_6r,$$
    so $pq | d_6$. So, $d_6\in \{pq^2, pqr\}$. If $d_6 = pq^2$, then $q^2 < d_6$, but $q^2$ does not appear before $d_6$ in $S'_N$, a contradiction. If $d_6 = pqr$, then $pr < d_6$, but $pr$ does not appear before $d_6$ in $S'_N$, again a contradiction. 
\end{itemize}
\end{enumerate}

\begin{prop}\label{p24} 
A number $N$ is large recurrent with $|L'_N|\ge 4$ and $(d_2, d_3) = (p, q)$ for some primes $p < q$ if and only if $N$ belongs to one of the following forms.
\begin{enumerate}
    \item $N = p^3q^2$ for $p < q < p^2$. In this case, $L'_N = \{q^2, p^2q, pq^2, p^3q, p^2q^2\}$ satisfies $U(q^2, p^2q, 0, p)$.
    \item $N = p^kq$ some for $k\ge 4$ and $p<q<p^2$.  In this case, 
    $$L'_N \ =\ \begin{cases}
    \{p^{{k/2}+1},p^{{k/2}}q,p^{{k/2}+2},\ldots,p^{k-1}q\} &\mbox{ if }2|k,\\
    \{p^{(k-1)/2}q,p^{{(k+3)/2}},p^{(k+1)/2}q,\ldots,p^{k-1}q\} &\mbox{ if }2\nmid k.
    \end{cases}$$
    Observe that $L'_N$ satisfies $U(p^{{k/2}+1},p^{{k/2}}q,0,p)$ and $U(p^{(k-1)/2}q,p^{{(k+3)/2}},0,p)$ for even and odd $k$, respectively.
    \item $N = pq^k$ for some $k\ge 4$ and $p < q$.  In this case, 
    $$L'_N \ =\ \begin{cases}
   \{pq^{k\over2},q^{{k\over2}+1},\ldots,q^k\} &\mbox{ if }2|k,\\
    \{q^{k+1\over2},pq^{k+1\over2},\ldots,q^k\} &\mbox{ if }2\nmid k.
    \end{cases}$$
    Observe that $L'_N$ satisfies $U(pq^{k/2},q^{{k/2}+1},0,q)$ and $U(q^{(k+1)/2},pq^{(k+1)/2},0,q)$ for even and odd $k$, respectively.
     \item  $N=pq^kr$ for some $k\ge 2$ and $p<q<pq^k<r$. In this case,  $L'_N=\{r,pr,qr,pqr,q^2r,\ldots,q^kr\}$ satisfies $U(r,pr,0,q)$.
\end{enumerate}
\end{prop}

Combining Propositions \ref{p23} and \ref{p24}, we obtain the following.

\begin{prop}\label{p100}
A number $N$ is large recurrent with $|L'_N|\ge 4$ if and only if $N$ belongs to one of the following forms.
\begin{enumerate}
     \item $N=p^k$ for some $k\ge 9$. In this case, $L'_N=\{p^{\lceil(k-1)/2\rceil+1}, p^{\lceil (k-1)/2\rceil+2},\ldots,p^{k-1}\}$ satisfies $U(p^{\lceil (k-1)/2\rceil+1},p^{\lceil (k-1)/2\rceil +2},p,0)$.
    \item $N=p^kq$ for some $k\ge 4$ and $q>p^k$. In this case, $L'_N=\{q,pq,p^2q,\ldots,p^{k-1}q\}$ satisfies $U(q,pq,p,0)$.
    \item $N=p^kq$ for some $k\ge 4$ and $p^{k-1} < q < p^k$. Then $$L'_N \ =\ \{p^k, pq, p^2q, \ldots, p^{k-1}q\}$$
    satisfies $U(p^k, pq, p, 0)$.
     \item $N = p^kq$ some for $k\ge 4$ and $p<q<p^2$.  In this case, 
    $$L'_N \ =\ \begin{cases}
    \{p^{{k/2}+1},p^{{k/2}}q,p^{{k/2}+2},\ldots,p^{k-1}q\} &\mbox{ if }2|k,\\
    \{p^{(k-1)/2}q,p^{{(k+3)/2}},p^{(k+1)/2}q,\ldots,p^{k-1}q\} &\mbox{ if }2\nmid k.
    \end{cases}$$
    Observe that $L'_N$ satisfies $U(p^{{k/2}+1},p^{{k/2}}q,0,p)$ and $U(p^{(k-1)/2}q,p^{{(k+3)/2}},0,p)$ for even and odd $k$, respectively.
    \item $N=p^4q$ with $p^2 < q < p^3$, $(p^5-q^2)|(p^2-q)$, and $(p^5-q^2)|(p^3-q)$. In this case, 
    $L'_N \ =\ \{pq, p^4, p^2q, p^3q\}$.
     \item $N = p^3q^2$ for $p < q < p^2$. In this case, $L'_N = \{q^2, p^2q, pq^2, p^3q, p^2q^2\}$ satisfies $U(q^2, p^2q, 0, p)$.
    \item $N = pq^k$ for some $k\ge 4$ and $p < q$.  In this case, 
    $$L'_N \ =\ \begin{cases}
   \{pq^{k\over2},q^{{k\over2}+1},\ldots,q^k\} &\mbox{ if }2|k,\\
    \{q^{k+1\over2},pq^{k+1\over2},\ldots,q^k\} &\mbox{ if }2\nmid k.
    \end{cases}$$
    Observe that $L'_N$ satisfies $U(pq^{k/2},q^{{k/2}+1},0,q)$ and $U(q^{(k+1)/2},pq^{(k+1)/2},0,q)$ for even and odd $k$, respectively.
     \item  $N=pq^kr$ for some $k\ge 2$ and $p<q<pq^k<r$. In this case,  $L'_N=\{r,pr,qr,pqr,q^2r,\ldots,q^kr\}$ satisfies $U(r,pr,0,q)$.
\end{enumerate}
\end{prop}

\subsection{The case $|L'_N|\le 3$}
If $|L'_N|\le 3$, then $\tau(N)\le 9$. We use the classifications of those $N$ from the introduction to obtain the following proposition
\begin{prop}\label{p220}
Let $p, q, r$ denote prime numbers and $k$ be some positive integer. 
A positive integer $N > 1$ is small recurrent with $|L'_N|\le 3$ if and only if $N$ belongs to one of the following forms.
\begin{enumerate}
    \item $N = p^k$ for some $k\le 8$. In this case, $L'_N = \{p^{\lceil (k-1)/2\rceil+1}, \ldots, p^{k-1}\}$ satisfies $U(p^{\lceil (k-1)/2\rceil+1}, p^{\lceil (k-1)/2\rceil+2}, p, 0)$.
    \item $N = pq$ for some $p < q$. In this case, $L'_N = \{q\}$.
    \item $N = pq^2$ for some $p < q$. In this case, $L'_N = \{pq , q^2\}$.
    \item $N = p^2q$ for some $p < q$. If $q < p^2$, then $L'_N = \{p^2, pq\}$. If $q > p^2$, then $L'_N = \{q, pq\}$.
    \item $N = pq^3$ for some $p < q$. In this case, $L'_N = \{q^2, pq^2, q^3\}$.
    \item $N = p^3q$ for some $p < q$. If $p < q < p^2$, then $L'_N = \{pq, p^3, p^2q\}$. If $p^3 < q$, then $L'_N = \{q, pq, p^2q\}$. If $p^2 < q < p^3$, then $L'_N = \{p^3, pq, p^2q\}$. 
    \item $N = p^2 q^2$ for some $p < q$. If $p < q < p^2$, then $L'_N = \{q^2, p^2q, pq^2\}$. The case $p^2 < q$ is impossible as it gives $L'_N = \{p^2q, pq^2, q^2\}$ and there is no integral solution $(a, b)$ to $apq^2 + bp^2q = q^2$.
    \item $N = pqr$ for some $p < q < r$. If $r > pq$, then $L'_N = \{r, pr, qr\}$. The case $r < pq$ is impossible as it gives $L'_N = \{pq, pr, qr\}$ and there is no integral solution $(a, b)$ to $apr + bpq = qr$. 
\end{enumerate}
\end{prop}
Combining Propositions \ref{p100} and \ref{p220}, we obtain Theorem \ref{m2}.


\section{Appendix}

\begin{proof}[Proof of Lemma \ref{l4}]
i) Since $r = aq + bp$ and $pq = ar + bq$, we know that $\gcd(a, b)|r$ amd $\gcd(a, b)|pq$, respectively. Hence, $\gcd(a, b) = 1$.

Since $r= aq + bp$ and $r$ is a prime, $p\nmid a$.

ii) Suppose that $k = \gcd(b, d_i) > 1$ for some $d_i\in S'_N$. If $d_i = d_2$, then $p|b$. Since $pq = ar + bq$, we get $p | a$, which contradicts $\gcd(a, b) = 1$. If $d_i = d_3$, then $q|b$. Since $r = aq + bp$, we get $q | r$, a contradiction. If $d_i > d_3$, then write 
$$\gcd(b, d_i)\ =\ \gcd(b, ad_{i-1} + bd_{i-2})\ \stackrel{\mbox{i)}}{=}\ \gcd(b, d_{i-1}),$$
which, by induction, gives $1 < \gcd(b, d_i) = \gcd(b, d_3)$, which has been shown to be impossible. 

iii) The claim holds for $i\le 4$. Let $d_i, d_{i+1}\in S'_N$ for some $i\ge 5$. We have
$$\gcd(d_i, d_{i+1})\ =\ \gcd(d_i, ad_i + bd_{i-1})\ \stackrel{\mbox{ii)}}{=}\ \gcd(d_i, d_{i-1}).$$
By induction, we obtain $\gcd(d_i, d_{i+1}) = 1$.

iv) The claim holds for $i\le 5$. Assume that it holds for all $i\le j$ for some $j\ge 5$. We show that it holds for $i = j+1$. We have
$$pq \ =\ ar + bq\ =\ a(aq + bp) + bq\ =\ (a^2+b)q + abp.$$
Hence, $p|(a^2+b)$. Write
$$d_{j+1} \ =\ ad_j + bd_{j-1}\ =\ a(ad_{j-1} + bd_{j-2}) + bd_{j-1}\ =\ (a^2+b)d_{j-1} + abd_{j-2}.$$
Since $p|(a^2+b)$ and $\gcd(p, ab) = 1$, we know that $p|d_{j+1}$ if and only if $p|d_{j-2}$. By the inductive hypothesis, $p|d_{j-2}$ if and only if $j-2\equiv 2\Mod 3$, or equivalently, $j+1\equiv 2\Mod 3$. By induction, we have the desired conclusion.  

v) The claim holds for $i\le 5$. Assume that it holds for all $i\le j$ for some $j\ge 5$. We show that it holds for $i = j+1$. That $pq = ar+bq$ implies that $q|a$. Write 
$$d_{j+1} \ =\ ad_j + bd_{j-1}.$$
By ii), $q|d_{j+1}$ if and only if $q|d_{j-1}$. By the inductive hypothesis, $q|d_{j-1}$ if and only if $j+1\equiv 1\Mod 2$. This completes our proof.
\end{proof}

\begin{proof}[Proof of Lemma \ref{l6}]
i) Same as the proof of Lemma \ref{l4} item i).

ii) Suppose, for a contradiction, that $\gcd(b, d_i) > 1$ for some $i\ge 3$. If $i = 3$, then $q|b$. We have $r = aq + bp$. Since $q|b$, we get $q|r$, a contradiction. If $i\ge 4$, write 
$$\gcd(b, d_i) \ =\ \gcd(b, ad_{i-1} + bd_{i-2})\ =\ \gcd(b, d_{i-1}).$$
By induction, $1 < \gcd(b, d_i) = \gcd(b, d_3)$, which has been shown to be impossible. 

iii) The claim holds for $i\le 4$. Pick $i\ge 5$. We have 
$$\gcd(d_i, d_{i+1}) \ =\ \gcd(d_i, ad_i + bd_{i-1})\ =\ \gcd(d_i, bd_{i-1})\ \stackrel{\mbox{ii)}}{=}\ \gcd(d_i, d_{i-1}).$$
By induction, we obtain $\gcd(d_i, d_{i+1}) = \gcd(d_4, d_5) = 1$. 

iv) Assume that $\gcd(a, d_i) > 1$ for some $i\ge 2$. If $i = 2$, then $p|a$, which contradicts the primality of $r$ and the linear recurrence $r = aq + bp$. If $i = 3$, then $q|b$, which contradicts the primality of $s$ and the linear recurrence $s = ar + bq$. Assume that $i\ge 4$. Write 
$$\gcd(a, d_i) \ =\ \gcd(a, ad_{i-1} + bd_{i-2})\ \stackrel{\mbox{i)}}{=}\ \gcd(a, d_{i-2}).$$
By induction, either $1 < \gcd(a, d_i) = \gcd(a, d_2)$ or $1 < \gcd(a, d_i) = \gcd(a, d_3)$, neither of which is possible. 

v) The claims holds for $i\le 3$. Pick $i\ge 4$ and suppose that 
$k = \gcd(d_i, d_{i+2}) > 1$. Since $d_{i+2} = ad_{i+1} + bd_{i}$, $k$ divides $ad_{i+1}$. By iii), $\gcd(k, d_{i+1}) = 1$, so $k|a$. However, $\gcd(a, d_i) > 1$ contradicts iv). 
\end{proof}


\ \\
\end{document}